\documentclass[11pt]{article}
\usepackage[a4paper,margin=1in]{geometry}
\usepackage{amsfonts,amssymb,amsmath}
\usepackage{graphicx}
\usepackage{psfrag}
\usepackage{calc}
\usepackage{epic}
\usepackage{subfigure}
\usepackage{comment}

\def\FRAC#1#2#3{\genfrac{}{}{}{#1}{#2}{#3}}
\def\half{{\mathchoice{\FRAC{1}{1}{2}}%
{\FRAC{2}{1}{2}}%
{\FRAC{3}{1}{2}}%
{\FRAC{4}{1}{2}}}}

\def\estP{P}

\def\estLambda{\Lambda}

\newlength{\noteWidth}
\long\def\notes#1{\ifinner
             {\tiny #1}
             \else
             \marginpar{\parbox[t]{\noteWidth}{\raggedright\tiny #1}}
             \fi}

\newcounter{rmnum}
\newenvironment{romannum}{\begin{list}{{\upshape (\roman{rmnum})}}{\usecounter{rmnum}
\setlength{\leftmargin}{10pt}
\setlength{\rightmargin}{8pt}
\setlength{\topsep}{-2pt}
\setlength{\itemsep}{2pt}
\setlength{\itemindent}{-1pt}
}}{\end{list}}

\def\limsup{\mathop{\rm lim\ sup}}
\def\liminf{\mathop{\rm lim\ inf}}

\def\arginf{\mathop{\rm arg\, inf}}

\def\bfmath#1{{\mathchoice{\mbox{\boldmath$#1$}}%
{\mbox{\boldmath$#1$}}%
{\mbox{\boldmath$\scriptstyle#1$}}%
{\mbox{\boldmath$\scriptscriptstyle#1$}}}}

\def\FRAC#1#2#3{\genfrac{}{}{}{#1}{#2}{#3}}

\def\ddt{{\mathchoice{\FRAC{1}{d}{dt}}%
{\FRAC{1}{d}{dt}}%
{\FRAC{3}{d}{dt}}%
{\FRAC{3}{d}{dt}}}}

\def\nablaIinv{\Xi}
\def\scaleS{S}

\def\bfmX{\bfmath{X}}

\def\bfmS{\bfmath{S}}

\def\bfmX{\bfmath{X}}

\def\A{ {\mathbb{A} }}
\def\C{ {\mathbb{C} }}

\def\R{ {\mathbb{R} }}

\def\L{ {\mathcal{L} }}

\def\Y{ {\mathcal{Y} }}

\def\Yd{\Y_{-\delta}}

\newenvironment{proof}[1][\scshape{Proof:}]
{\begin{trivlist} \item[\hskip \labelsep {\small #1}]}
{\begin{flushright}$\blacksquare$\end{flushright}\end{trivlist}}

\def\Sec#1{Sec.~\ref{#1}}
\def\Fig#1{Fig.~\ref{#1}}
\def\Thm#1{Theorem~\ref{#1}}
\def\eq#1{eq.~\eqref{#1}}

\newtheorem{theorem}{Theorem}

\newtheorem{lemma}{Lemma}
\newtheorem{proposition}{Proposition}

\newtheorem{assumption}{Assumption}

\begin{document}
\title{Large deviation asymptotics for busy periods}
\date{\today}
\author{
Ken R. Duffy\thanks{Hamilton
Institute, National University of Ireland Maynooth, 
Ireland. 
E-mail: \texttt{ken.duffy@nuim.ie}}
\and 
Sean P. Meyn\thanks{Dept. of Electrical and Computer Engineering,
University of Florida, Gainesville, FL 32611-6200 U.S.A.  
E-mail: \texttt{meyn@ece.ufl.edu}}
}

\maketitle

\begin{abstract}
The busy period for a queue is cast as the area swept under the
random walk until it first returns to zero, $B$. Encompassing
non-i.i.d. increments, the large-deviations asymptotics of $B$ is
addressed, under the assumption that the increments satisfy standard
conditions, including a negative drift.  The main conclusions provide
insight on the probability of a large busy period, and the manner
in which this occurs:
\begin{romannum}
\item
The scaled probability of a large busy period has the asymptote, for any $b>0$, 
\begin{align*}
\lim_{n\to\infty} \frac{1}{\sqrt{n}} \log P(B\geq bn) = -K\sqrt{b},
\end{align*}
\begin{align*}
\hbox{
where}
\quad
K = 2 \sqrt{-\int_0^{\lambda^*} \Lambda(\theta) \, d\theta}\,, \quad
\hbox{with $\lambda^*=\sup\{\theta:\Lambda(\theta)\leq0\}$,}
\end{align*}
and with $\Lambda$ denoting the scaled cumulant generating function
of the increments process.

\item
The most likely path to a large swept area is found to be a simple rescaling
of the path  on $[0,1]$ given by,
\[
\psi^*(t) = -\Lambda(\lambda^*(1-t))/\lambda^*\, .
\]
In
contrast to the piecewise linear most likely path leading the random
walk to hit a high level, this is strictly concave in general.
While these two most likely paths have very different forms, their
derivatives coincide at the start of their trajectories, and at
their first return to zero. 

\end{romannum}
These results partially answer an open problem of Kulick and Palmowski
\cite{Kulick11} regarding the tail of the work done during a busy
period at a single server queue.
The paper concludes with applications of these results to the
estimation of the busy period statistics $(\lambda^*, K)$ based on
observations of the increments, offering the possibility of estimating
the likelihood of a large busy period in advance of observing one.

\end{abstract}

\section{Introduction}
Consider $\bfmS=\{S_k : k\ge 0\}$, a random walk that starts at zero
and has (not necessarily i.i.d.) increments process
$\bfmX=\{X_n: n\ge0\}$: 
\begin{align}
\label{eq:lindley}
S_0:=0 \text{ and } S_{k} = \sum_{i=1}^k X_i \text{ for } k\geq1.
\end{align} 
Define the stopping time and stopped variable:
\begin{align}
\label{eq:defB}
\tau := \inf\{k\geq1:S_k\le0\} \text{ and } B = \sum_{i=1}^\tau S_i.
\end{align}
The tail behavior of $B$ (and related random variables) is of
interest in several apparently distinct fields from queueing systems,
to percolation, to insurance \cite{Kulick11,CTCN,DuffyMeyn10}.

By simple rescaling arguments, should the following limit exist,
it must have this form:
\begin{align}
\label{eq:K}
\lim_{n\to\infty} \frac {1}{\sqrt{n}} \log P(B\geq nb) = - K\sqrt{b}
\text{ for some } K\geq 0.
\end{align}
This limit is established with $K>0$ in \cite{blaglymey11} for the
particular case of the the M/M/1 queue. By extending results in
\cite{DuffyMeyn10} from a fixed terminal point to the random terminal
point $\tau$ via the infinite time-horizon sample path Large Deviation
Principle setup in \cite{Ganesh04}, \Thm{thm:main} establishes the
limit \eq{eq:K} for a broad class of non-long-range dependent,
light-tailed arrivals processes, providing a formula for $K$. 
The scaled Cumulant Generating Function (sCGF)
associated with the
scalar LDP is denoted
\begin{align}
\label{eq:scgf}
\Lambda(\theta) = \lim_{k\to\infty} \frac 1k \log E(\exp(\theta S_k))
\text{ for } \theta\in\R.
\end{align}
Identifying $\lambda^*=\sup\{\theta:\Lambda(\theta)\leq
0\}$, 
\Thm{thm:main} shows that
under Assumptions \ref{ass:main} and \ref{ass:main2}, 
\begin{equation}
K = 2 \sqrt{-\int_0^{\lambda^*} \Lambda(\theta) \, d\theta}. 
\label{e:Kformula}
\end{equation}
Its proof establishes that the most likely path to a large swept
area is strictly concave. The most likely path that first returns
to zero at $t=1$ is identified to be
\begin{align*}
\psi^*(t) = -\frac{1}{\lambda^*} \Lambda(\lambda^*(1-t)) 
\text{ for } t\in[0,1],
\end{align*}
with all other most likely paths being simple rescalings of $\psi^*$, as illustrated in \Fig{fig:Lambdapsi}.
This is in contrast to the most likely path for the random walk to
hit a high level, which is known to be piecewise linear, e.g.
\cite{Anantharam89, Ganesh02}.  

\begin{figure}
\centerline{
\includegraphics[width=.85\columnwidth]{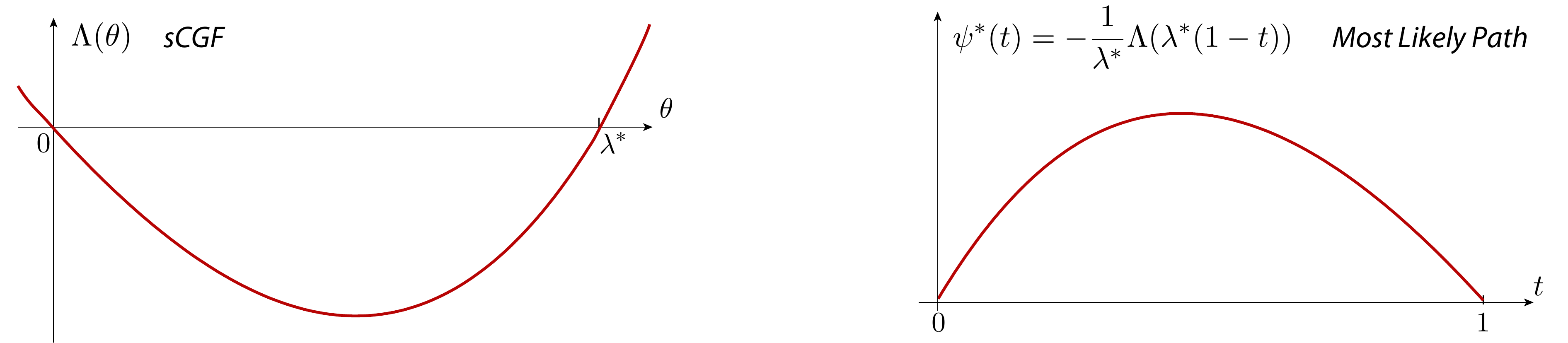}
}
\caption{Relationship between $\Lambda$ and $\psi^*$}
\label{fig:Lambdapsi}
\end{figure}

\begin{figure}
\begin{center}
\includegraphics[width=1.0\columnwidth]{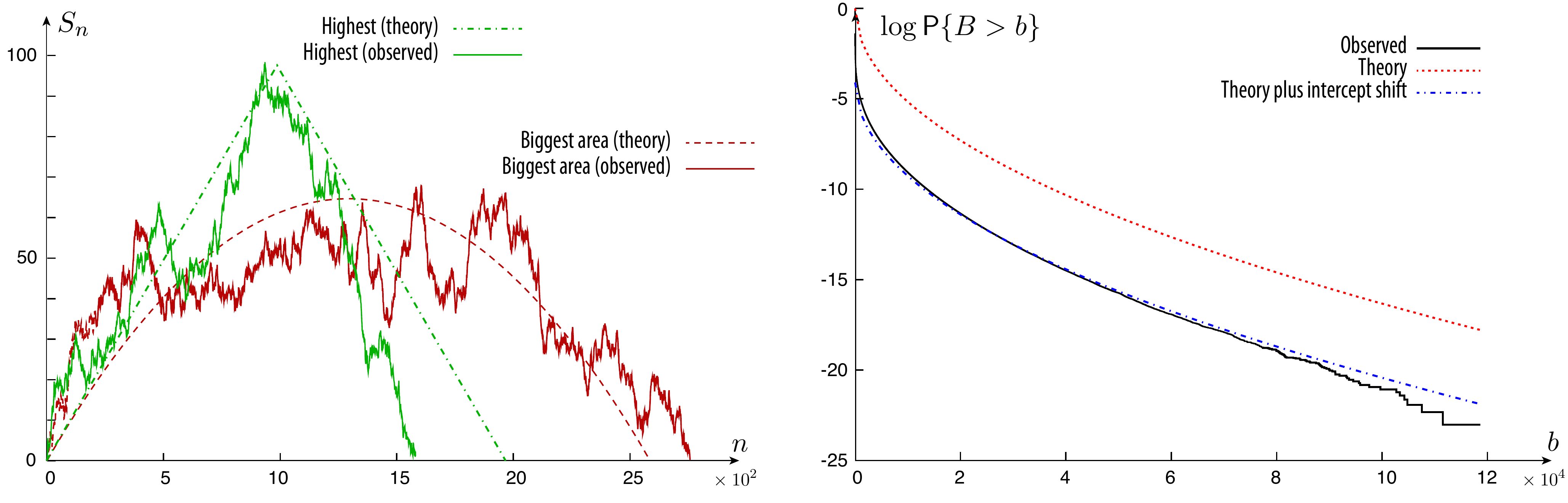}
\caption{I.i.d. Gaussian(-1/10,1) increments. 
Simulated results from ten billion paths. 
(Left panel)
Four paths displayed. Simulated path with largest swept area
$111,524.13$ and, conditioned on that area, the most likely path
to give rise to it as predicted by \Thm{thm:main}, which is concave.
Also shown is the simulated path that reaches the greatest height
$98.36$ and, conditioned on that height, the most likely path to
give rise to it as deduced from \Thm{thm:goc}, which is piecewise
linear. (Right panel) Logarithm of the empirical probability
that the area under busy period exceeds a given level compared
with theoretical prediction of $-K\sqrt{b}$ and (as the
asymptote does not capture the prefactor) the linear shift
$-K\sqrt{b}+\kappa$, where $\kappa$ is chosen to match the
offset of the empirical observation}
\label{fig:gauss}
\end{center}
\end{figure}

As an illustrative example, consider a random walk with i.i.d
Gaussian increments for which everything is calculated in closed
form in \Sec{sec:exgauss}. Ten billion busy period paths were
simulated and, of these, the one with the largest swept area as
well as the one that attained the greatest height were recorded.
In addition to plotting these paths in \Fig{fig:gauss}, conditioned
on these values the most likely paths to these events are shown.
The distinct shapes of the paths to these two unlikely events is
apparent. Note that the terminal times of the theoretically predicted
paths are deductions of the conditioned area and height, respectively,
and are not constrained explicitly. Also shown is the logarithm of
the empirical probability with which the busy period exceeded a
given value as compared to the large deviation estimate. Up to a
constant prefactor, the approximation inferred from \Thm{thm:main}
is remarkably accurate.

These results provide a partial answer to Open Problem 3.1 of
\cite{Kulick11} regarding precise asymptotics for the probability
that $B$ is large, by identifying the associated rough asymptotics
and showing that most likely paths are typically strictly concave.
They also reveal a lacuna in \cite[Theorem 4.1]{Borovkov03} where
the most likely path is assumed to be piecewise linear.

We conclude the paper in \Sec{sec:conc} with a discussion of the
practical utility of these results. Given observations of the
increments process $\bfmX$, estimates of the key quantities
$(\lambda^*,K)$ can be created that, under additional restrictions
on $\bfmX$, can be shown to satisfy a large deviation principle.
This offers the possibility of accurately estimating the likelihood
of a long busy period in advance of observing one.

\section{Functional setup}

The framework for analysis used in this paper  was first developed for weak
convergence of probability measures \cite{Muller68,Borovkov73,Whitt72},
and subsequently used in the context of sample path large deviations,
e.g. \cite{Deuschel89,Majewski00,Wischik01,Ganesh02,Ganesh04,Duffy08a}.
In particular, Ganesh and O'Connell \cite{Ganesh02} employed this
setup to establish an infinite time horizon version of Anantharam's
result \cite{Anantharam89}, proving that the most likely path to
exceed a high level for a random walk with negative drift is piecewise
linear on the scale of large deviations.

Let $\C[0,\infty)$ denote the collection of real-valued continuous
functions on $[0,\infty)$. Let $\A[0,\infty)$ denote the collection
of the integrals of functions that are elements of $\L^1[0,x)$ for
all $x>0$ (for example, see Riesz and Sz.-Nagy \cite{Riesz55}).
For each $r\in\R$, define the space
\begin{align}
\label{e:Yr}
\Y_r:=\left\{\phi\in\C[0,\infty):\lim_{t\to\infty}
	\frac{\phi(t)}{(1+t)} = r \right\}
\end{align}
and equip it with the topology induced by the norm
\begin{displaymath}
\|\phi\| =  \sup_{t\ge0} \left|\frac{\phi(t)}{1+t}\right|.
\end{displaymath}

Define the polygonal sample paths 
\begin{align*}
S_n(t):= \frac 1n S_{[nt]}+\left(t-\frac{[nt]}{n}\right)X_{[nt]+1},
        \;\mbox{for}\;t\in[0,\infty).
\end{align*}
The sample path process $\{S_n(\cdot)\}$ is known to satisfy the
LDP in $\Y_r$ for a broad class of non-long-range dependent,
non-heavy-tailed random walks. For example, see \cite[Theorem
1]{Ganesh02}, where \cite[Theorem 2]{Dembo95} provides general
mixing and uniform tail exponent conditions under which the
prerequisites of this theorem hold. This encompasses increment
processes that are Harris recurrent Markov chains, subject to a
Foster-Lyapunov drift condition~\cite{konmey05a}. The existence
of such an LDP will be the primary assumption in our proof of
\eq{eq:K}.

\section{Tail asymptotics}

We shall make two assumptions. The first is the existence of a
sample path LDP for the random walk that ensures the walk has
negative drift, but a possibility of becoming positive on the scale
of large deviations.
\begin{assumption}
\label{ass:main}
The sample path process $\{S_n(\cdot)\}$ satisfies the LDP in $\Yd$, some $\delta >0$, with rate function
\begin{align}
\label{eq:Iinfty}
I_\infty(\phi)=
\begin{cases}
\int_0^\infty I(\dot\phi(t))dt & \;\mbox{if}\;\phi\in\A[0,\infty)\cap\Yd,\\
	+\infty & \text{ otherwise},
\end{cases}
\end{align}
where $I$ is the good, strictly convex rate function associated
with the random walk $\{S_n(1)\}$ and $I(x)<\infty$ for some $x>0$.
\end{assumption}
Our identification of the most likely path to a large swept area
will have a surprising relationship with the most likely path to
exceed a high level, so we recall the following result of Ganesh
and O'Connell.
\begin{theorem}[\cite{Ganesh02}]
\label{thm:goc}
Under Assumption~\ref{ass:main}, the following exist and are
non-negative
\begin{align} 
\label{e:xstar}
x^* = \arginf_{x>0} \frac{1}{x} I(x) 
\text{ and }
\lambda^* = \frac{1}{x^*} I(x^*) = 
	\sup\{\theta:\Lambda(\theta)\leq 0\}.
\end{align}
Moreover, for any $h>0$,
\begin{align}
\label{eq:glynnwhitt}
\lim_{n\to\infty} \frac 1n \log P\left(\sup_k S_k \geq nh\right) = 
	-h \lambda^*,
\end{align}
while the most likely path to this event is
\begin{align}
\label{eq:varphi}
\varphi^*(t)  = 
	\begin{cases}
	x^*t  & \text{ if } t<h/x^*,\\	
	h+(t-h/x^*)(-\delta) & \text{ if } t\geq h/x^*.
	\end{cases}
\end{align}
\end{theorem}
Note that the probability of sweeping a large area, \eq{eq:K},
decays on a slower scale than the probability of hitting a high
height, \eq{eq:glynnwhitt}.

For the supremum of a random walk, $\lambda^*$ determines the rate
of decay of the probability of hitting a high level as shown in
\eq{eq:glynnwhitt}. For busy periods, it will play a new and
surprising r\^ole in the characterization of the most likely path
to a large swept area. The inverse of $\nabla I$ will prove 
central to the development that follows. To that end, our second
assumption is the following regularity condition.
\begin{assumption}
\label{ass:main2}
The rate function $I$ is continuously differentiable on an interval
that contains $[-\delta,x^*]$.
\end{assumption}
This assumption justifies the definition, 
\begin{align}
\label{e:nablaIinv}
\nablaIinv(r) =: (\nabla I)^{-1}(r),\quad 
	\text{$r\in\R$, whenever the inverse exists.}
\end{align}
The inverse exists, so that $\nablaIinv$   is finite-valued, on an interval that contains
$[\nablaIinv(-\delta),\nablaIinv(x^*)]=[0,\lambda^*]$.  That
$\lambda^*$ is significant here stems from the second part of
the following lemma.
\begin{lemma}
\label{lem:lambdastaralt}
The scalar $\lambda^*$ defined in \eq{e:xstar} satisfies 
\begin{align*}
\lambda^*= \nabla I(x^*). 
\end{align*}
Moreover, it is the unique positive solution of
\begin{align}
\label{e:lambda3}
\int_0^1\nablaIinv(\lambda^*s)\, ds =0.
\end{align}
\end{lemma}
\begin{proof}
The identity $\nabla I(x^*) = I(x^*)/x^*$ follows from the first-order
optimality condition for $x^*$, based on its definition in \eq{e:xstar}.

Clearly $\lambda^*$ is positive  as $x^*$ is positive and $I(x)>0$
for all $x>-\delta$. To see that $\lambda^*$ thus defined is a
solution of \eq{e:lambda3}, direct substitution, change of
variables and integration by parts suffices. To see it is unique,
note that it is equivalently characterized as any positive solution
of
\begin{align*}
\int_0^{\lambda^*}\nablaIinv(s)\, ds =0.  
\end{align*}
As $I$ is strictly convex, $\nablaIinv(x)$ is strictly increasing
when finite. At $x=0$ we have that $\nablaIinv(0)=-\delta<0$,
so this equation only has one positive solution.
\end{proof}

We will use $\lambda^*$ to define the most likely path to sweep
an area over $[0,1]$:
\begin{align}
\label{eq:psi*}
\psi^*(t) =
	\begin{cases}
	\displaystyle\int_0^t \nablaIinv\left(\lambda^*(1-s)\right) ds& 
		\text{ for } t\in[0,1]\\
	(1-t)\delta & \text { for } t\geq 1.
	\end{cases}
\end{align}
The most likely path to sweep any other area will be a simple
rescaling of this solution.
This path $\psi^*$ is strictly concave on $[0,1]$ as $I$ is strictly convex.
On this interval it is of the form found in \cite{DuffyMeyn10} in
the analysis of simulation of queues: on differentiating each side
of \eq{eq:psi*}, it follows that the path satisfies the simple
differential equation,
\begin{align}
\label{e:psi*dot}
\nabla I \, \left(\ddt\psi^*(t) \right)  = \lambda^*(1-t) 
	\text{ for } t\in[0,1].
\end{align}
Note that, thus defined, we have that
\begin{align}
\label{eq:beginend}
\ddt\psi^*\, (0) = \nablaIinv(\lambda^*)=x^*
\quad
\text{ and }
\quad
\ddt\psi^*\, (1) = \nablaIinv(0)=-\delta.
\end{align}
That is, remarkably, the most likely paths to sweeping a large area,
$\psi^*$ in \eq{eq:psi*}, and to exceeding a large height, $\varphi^*$
in \eq{eq:varphi}, both start and end with identical derivatives,
but are distinct in-between. Before providing the main result, we
establish the following   characterizations of $\psi^*$. 
\begin{proposition}
\label{lem:psistaralt}
The following hold for $t\in[0,1]$:
\begin{romannum}
\item
In contrast to \eq{eq:psi*}, a non-integral representation is obtained
in terms of the rate function,
\begin{align}
\label{eq:psi*alt}
\psi^*(t) =
\frac{I\left( \nablaIinv(\lambda^*(1-t)) \right)}{\lambda^*}
                - \nablaIinv(\lambda^*(1-t)) (1-t).
\end{align}

\item In terms of the sCGF,
\begin{align}
\label{eq:psi*alt1}
\psi^*(t) = -\frac{1}{\lambda^*} \Lambda(\lambda^*(1-t)).
\end{align}
and hence, 
\begin{align*}
\ddt\psi^*(t) = \nabla\Lambda(\lambda^*(1-t))
\end{align*}  
\end{romannum}
\end{proposition}

\begin{proof} 
The identity $\nabla\Lambda(\theta) = \nablaIinv(\theta)$ can be
established for $\theta\in[0,\lambda^*]$ based on convex duality
(e.g,  \cite[Proposition 11.3]{Rockafellar98}).
The representation in \eq{eq:psi*alt1} then follows by integration
using $\psi^*(0)=0$. The characterization in \eq{eq:psi*alt} can
be obtained from \eq{eq:psi*alt1} noting that
$\Lambda(\theta)=\theta\,\nablaIinv(\theta)-I(\nablaIinv(\theta))$ for
$\theta\in[0,\lambda^*]$.
\end{proof}

A simpler expression for $\lambda^*$ is obtained when  $\Lambda$
is symmetric.  This symmetry can be interpreted as asymptotic
reversibility of the underlying walk.
\begin{proposition}
\label{lemma:rev}
Recall that $\lambda^*>0$ is a zero of $\Lambda$, $-\delta<0$ is a
zero of $I$, and $x^*>0$ is the slope given in Theorem~\ref{thm:goc}.
Under the symmetry condition,
\begin{align}
\label{sym:lambda}
\Lambda(\theta) = \Lambda(\lambda^*-\theta) \text{ for all } 
	\theta\in [0,\lambda^*],
\end{align}
then these parameters are related as follows,
\begin{align}
\label{eq:lambdarev}
x^* = \delta 
\text{ and }
\lambda^* = \nabla I(\delta) = I(\delta)/\delta = 2\nabla I(0).
\end{align}
Moreover, if \eq{sym:lambda} holds for all $\theta\in\R$, then 
\begin{align}
\label{eq:lambdarev2}
\lambda^* = \frac{I(x)-I(-x)}{x} \text{ for all } x\neq 0 \text{ such that }
I(x) <\infty.
\end{align}
\end{proposition}

\begin{proof}
If $\Lambda(\theta) = \Lambda(\lambda^*-\theta)$ for $\theta\in [0,\lambda^*]$, 
the path $\psi^*(t)$ is symmetric with $\psi^*(t)=\psi^*(1-t)$.
Thus $\ddt\psi^*(t)=-\ddt\psi^*(1-t)$.
In particular, evaluating this at $t=0$ and using 
\eq{eq:beginend} gives $x^*=\ddt\psi^*(0)= -\ddt\psi^*(1)=\delta$.
Hence by \eq{e:xstar}, $\lambda^* = \nabla I(\delta)
= I(\delta)/\delta$. Equating the derivatives at $t=1/2$ gives
$2\,\nablaIinv(\lambda^*/2)=0$ so that $\lambda^* = 2\nabla I(0)$.

If \eq{sym:lambda} holds for all $\theta$, then
\begin{align*}
I(-x) 
&= \sup_{\theta\in\R} \left(\theta(-x) -\Lambda(\lambda^*-\theta) \right)
= \sup_{\theta\in\R} \left((\lambda^*-\theta)(-x) -\Lambda(\theta)\right) 
= I(x) -\lambda^*x \text{ for all } x.
\end{align*}
This gives \eq{eq:lambdarev2} for all $x\in\R$ such that $I(x)<\infty$.
\end{proof}

For example, eq. \eqref{sym:lambda} is satisfied if $\bfmX$ is
i.i.d., $E(\exp(\theta X_1)$ is finite in a neighborhood of the
origin and $P(X_1=x)/P(X_1=-x)=\exp(-\lambda^* x)$ for all $x$, as
holds for $X_1$ Gaussian or Bernoulli-$\{-C,+C\}$. This condition,
however, extends beyond i.i.d. increments processes and in
\Sec{sec:examples} we present an example where eq. \eqref{sym:lambda}
is satisfied for a Markov chain $\bfmX$.

Armed with these assumptions, definitions and characterizations of
the path $\psi^*$, we now prove the main result.

\begin{theorem}
\label{thm:main}
Under the above assumptions, for any $b>0$,
\begin{align*}
\lim_{n\to\infty} \frac 1n \log P(B\geq n^2b) = - K\sqrt{b}.
\end{align*}
where $K$ is given in \eqref{e:Kformula}.
Moreover,  
\begin{romannum}
\item 
The most likely asymptotic value of $\tau/n$ leading to $\{B\ge
n^2b\}$ is 
\begin{align*}
a =:  \frac{2 \lambda^*}{K}  \sqrt{b}.
\end{align*}
\item
The rescaled, most likely asymptotic path of $S_n(\cdot)$ is
\begin{align}
\label{eq:psi*b}
\psi^*_b(t)=a\psi^*(t/a) 
	= -2\frac{\sqrt{b}}{K}
	\,
	\Lambda\left(\lambda^*
	-\frac{t}{2}\frac{K}{\sqrt{b}}\right),
\end{align}
which is strictly concave on $[0,a]$.

\end{romannum}
\end{theorem}

\begin{proof}

The method of proof is to construct a collection of open sets,
$\{B_\epsilon\}$, and a closed set $F$ in $\Yd$ such that for all
$\epsilon>0$ sufficiently small
\begin{align*}
\{S_n(\cdot)\in B_\epsilon\}
	\subset
	\{B\geq bn^2\}
	\subset \{S_n(\cdot)\in F\}
\end{align*}
and 
\begin{align*}
\lim_{\epsilon\to0}\liminf_{n\to\infty}
	 \frac 1n \log P(S_n(\cdot)\in B_\epsilon)
	= 
	\limsup_{n\to\infty} \frac 1n \log P(S_n(\cdot)\in F).
\end{align*}

Define
\begin{align*}
F := \left\{\phi\in\Yd: \int_0^\infty \max(\phi(t),0) dt \geq b\right\}.
\end{align*}
If $B\geq bn^2$, then $S_n(\cdot)\in F$ as
\begin{align*}
\int_0^\infty \max(S_n(t),0) dt 
	\geq \int_0^{n\tau} S_n(t) dt 
	= \frac{1}{n^2} \sum_{i=1}^{\tau} S_i 
	= \frac{1}{n^2} B \geq b.
\end{align*}
The set $F$ is closed as $\phi\mapsto\max(\phi,0)$ is Lipschitz
continuous from $\Yd\to\Y_0$ and, as for any $\phi\in\Yd$ we have
$\phi(t)<0$ for all $t$ sufficiently large, integration is also
continuous (e.g. \cite[Theorem 11.5.1]{Whitt02}). Thus we can use
the LDP upper bound to obtain
\begin{align*}
\limsup_{n\to\infty} \frac 1n \log P(B\geq bn^2) 
	&\leq
	\limsup_{n\to\infty} \frac 1n \log P(S_n(\cdot)\in F)\\
	&= -\inf\{I_\infty(\phi):\phi\in F\}
	= -\inf_{t>0}\inf\left\{I_\infty(\phi):
		\int_0^t\max(\phi(s),0) \, ds\geq b\right\}.
\end{align*}
As $\lim_{t\to\infty} \phi(t)/(1+t) = -\delta<0$ for all $\phi\in\Yd$,
this infimum over $t$ is attained at some finite $t$. Consider this
inner functional infimum for fixed $t$:
\begin{align}
\text{{\bf minimize} } \quad& I_\infty(\phi) \nonumber\\
\text{{\bf subject to} } \quad& 
\phi\in L^+[0,t] \text{ and } \int_0^t\max(\phi(t),0)\geq b.
\label{eq:optimization}
\end{align}
This functional optimization problem is closely related to \cite[eq.
(6)]{DuffyMeyn10}, where one can identify $b$ with $z$. Mild
alterations to Proposition 7 therein shows that with
$b=a^2\int_0^1\psi^*(t)dt$ for some $a>0$, with $\psi^*(t)$ defined
in \eq{eq:psi*}, this infimum 
occurs at any $t\geq a$ for
which it transpires that $\psi^*_b$, defined in \eq{eq:psi*b},
is the optimizer. This essentially occurs as
\begin{align*}
\psi^*_b = \arginf\left\{\int_0^a I(\dot\phi(t))dt : \int_0^a\phi(t) 
	= a^2 \int_0^1\psi^*(t)dt\right\}
\end{align*}
The quantity $\lambda^*$ in the definition of $\psi^*$ arises as a
scalar Lagrange multiplier \cite{DuffyMeyn10}. 
Note that $\psi^*_b(a)=0$, 
and thus the
most likely path to sweep a rescaled area $b$ satisfies $\tau/n\approx
a = \sqrt{b/\int_0^1\psi^*(t)dt}$. Using Proposition~\ref{lem:psistaralt} 
we have that 
\begin{align*}
\int_0^1 \psi^*(t) \, dt 
	= -\frac{1}{\lambda^*}\int_0^1 \Lambda(\lambda^*(1-t))\, dt
	= -\frac{1}{(\lambda^*)^2}\int_0^{\lambda^*} \Lambda(\theta)\, d\theta
\end{align*}
and thus the expression for $a$ in the statement follows.

What remains to be shown is that there is a coincident lower bound.
Let $\psi^*$ be the unique solution of \eq{eq:psi*} and define
$\psi^*_b$ using \eq{eq:psi*b}. The path $\psi^*_b(t)$ starts at
$0$ ends at $a$,
and is the optimal path that sweeps an area of
$\int_0^a\psi^*_b(t)dt=a^2\int_0^1\psi^*(t)dt=b$.  Let
\begin{align*}
B_\epsilon = B^1_\epsilon\cap B^2_\epsilon
\end{align*}
where, with $e(t)=\epsilon(1+t)$,
\begin{align*}
B^1_\epsilon = \left\{\phi:\inf_{t\in[0,\epsilon]}(\phi(t)-t\epsilon)>0\right\}
\text{ and }
B^2_\epsilon = \left\{\phi: \|\phi-(\psi^*_b+e)\|<\epsilon\right\}
\end{align*}
Both $B^1_\epsilon$ and $B^2_\epsilon$ are open by
construction and, for $\epsilon$ sufficiently small, their intersection
is non-empty. As defined, $S_n(\cdot)\notin B^1_\epsilon$ as $S_n(0)=0$,
but this is not significant as we can use an exponentially equivalent
representation with $S_n(0):=1/n$. Thus if $S_n(\cdot)\in B_\epsilon$
with $S_n(0):=1/n$, then $\{B>bn^2\}$. As $B_\epsilon$ is open for
all $\epsilon$, we can use the LDP lower bound
\begin{align*}
\liminf_{n\to\infty} \frac 1n \log P(B\geq bn^2) 
	&\geq
	\lim_{\epsilon\to0}\liminf_{n\to\infty} 
	\frac 1n \log P(S_n(\cdot)\in B_\epsilon)\\
	&\geq
	\lim_{\epsilon\to0}\liminf_{n\to\infty} 
	\frac 1n \log P(S_n(\cdot)\in B_\epsilon)\\
	&= -\lim_{\epsilon\to0}\inf\{I_\infty(\phi):\phi\in B_\epsilon\}
	= -I_\infty(\psi^*_b).
\end{align*}

To evaluate $K$, note that
\begin{align}
\label{eq:Kalt}
\lim_{n\to\infty} \frac 1n \log P(B\geq bn^2) 
	= - I_\infty(\psi^*_b)
	= 
	-\frac{\int_0^1 I(\ddt\psi^*(t)) dt}{\sqrt{\int_0^1\psi^*(t)dt}} 
	\sqrt{b}.
\end{align}
Using
$\Lambda(\theta)=\theta\,\nablaIinv(\theta)-I(\nablaIinv(\theta))$ for
$\theta\in[0,\lambda^*]$, integration by parts and the fact that
$\nabla\Lambda(\lambda^*)=0$, we have that
\begin{align*}
\int_0^1 I(\dot\psi^*(t)) \, dt
	= -\frac{2}{\lambda^*}\int_0^{\lambda^*} \Lambda(\theta) \, d\theta.
\end{align*}
Inserting this into \eq{eq:Kalt} in conjunction with the expression
for $\int_0^1 \psi^*(t) \, dt$ given above obtains the expression for
$K$ in \eqref{e:Kformula}.

\end{proof}

\section{Examples}
\label{sec:examples}

\Thm{thm:main} provides a mechanism for calculating the most likely
path to a large busy period $B$ as well as the exponent $K$. We shall perform
this calculation for illustrative examples: with i.i.d Gaussian
increments where $\lambda^*$, $\psi^*$ and $K$ can all be determined
in closed form; with Bernoulli$\{-1,+C\}$ increments, which includes
M/M/1 queue lengths, where explicit expressions of $\lambda^*$ are
not always possible, but $\psi^*$ can be written in terms of it and
$K$ must always be calculated numerically; and, finally, for
increments with Markovian dependencies where $\lambda^*$ and $\psi^*$
can be determined in closed form, but $K$ must be identified
numerically.

For each of the examples ten billion paths were simulated. As well
as recording the logarithm of the frequency with which $B$ exceeded
$b$ as a function of $b$, the largest swept area and highest paths
were logged for comparison with the theoretically predicted most
likely paths. For comparison with observations, \Thm{thm:main}
says that given we observe $B=n^2b$, on the scale of large deviations
the most likely time taken to generate the area is
\begin{align*}
\tau \approx na 
	= n 2\sqrt{b}/K
	=\tau^*_n \quad 
	    \text{where}
    \quad 
    \tau^*_n =: \half K /B^{-1/2} 
\end{align*}
and the most likely path is then 
\begin{align*}
\scaleS_i =: n S_n\left(i/n\right) 
	\approx \psi^*_{B/n^2}\left(i/n\right) 
	= na \psi^*\left(i/(na)\right)  \, .
\end{align*}
Thus, since $\tau^*_n = an$, for large swept area we have the approximation
\begin{align}
\label{eq:bigpathapprox}
\scaleS_i
    \approx   \tau^*_n \psi^*(  i/\tau^*_n). 
\end{align}
This most likely path is solely parameterized from the observations
by the value $B$. In particular, note that given $B$ the time $\tau$
is determined, so in the comparisons for the simulation results that follow,
the length of the most likely path is not explicitly fit to data. 

Similarly, \eq{eq:K} leads to the approximation
\begin{align}
\label{eq:rfapprox}
\log P(B\geq b) \approx -K\sqrt{b}, \text{ for large } b.
\end{align}

\smallbreak

For contrast, we also record the path that reaches the highest
height and use \Thm{thm:goc} for comparison.  Given a path
$S_n$ that reaches a height $H$, $S_n(\cdot)$ reaches $H/n$ and
therefore
\begin{align}
\label{eq:highpathapprox}
\scaleS_i = n S_n(i/n) \approx n \varphi^*(i/n) =
	\begin{cases}
	x^*i  & \text{ if } i<H/x^*,\\	
	H+(i-H/x^*)(-\delta) & \text{ if } i\geq H/x^*.
	\end{cases}
\end{align}
Again, given $H$, the time at which this most likely path returns
to zero is completely determined.

\subsection{Gaussian increments}
\label{sec:exgauss}

Let $\bfmX$ be i.i.d. Gaussian, with $X_0$ having mean $-\delta<0$
and variance $\sigma^2$. As $\bfmX$ is i.i.d., we have that
\begin{align*}
\Lambda(\theta) = \frac{\sigma^2}{2} \theta^2 - \delta \theta.
\end{align*}
Using \Thm{thm:main} we obtain the following explicit expressions
\begin{align*}
\lambda^* = \sup\{\theta:\Lambda(\theta)\leq 0\} = \frac{2\delta}{\sigma^2},
\qquad
\psi^*(t) = -\frac{1}{\lambda^*}\Lambda(\lambda^*(1-t)) = \delta t (1-t)\\
\text{ and } 
K =  2 \sqrt{-\int_0^{\lambda^*} \Lambda(\theta) \, d\theta} 
  =\frac{\delta^{3/2}}{\sigma^2}\sqrt{\frac{8}{3}}.
\end{align*}
The most likely path to $\{B\geq bn^2\}$ for large $n$ can also be
determined to be
\begin{align*}
\psi^*_b(t) = \delta t\left(1-t\sqrt{\frac{\delta}{6b}}\right).
\end{align*}

A demonstration of these results appears in \Fig{fig:gauss}. The
highest and largest swept area paths are compared with the most
likely conditioned on these quantities using the approximations in
\eq{eq:bigpathapprox} and \eq{eq:highpathapprox}. The distinct
nature of two paths to these two unlikely events is evident.  The
empirical likelihood of a large deviation is shown along with that
from the approximation above, which gives remarkably good agreement
up to a constant prefactor.

\subsection{Bernoulli $\{-1,+C\}$ increments}

\begin{figure}
\begin{center}
\includegraphics[width=1.00\columnwidth]{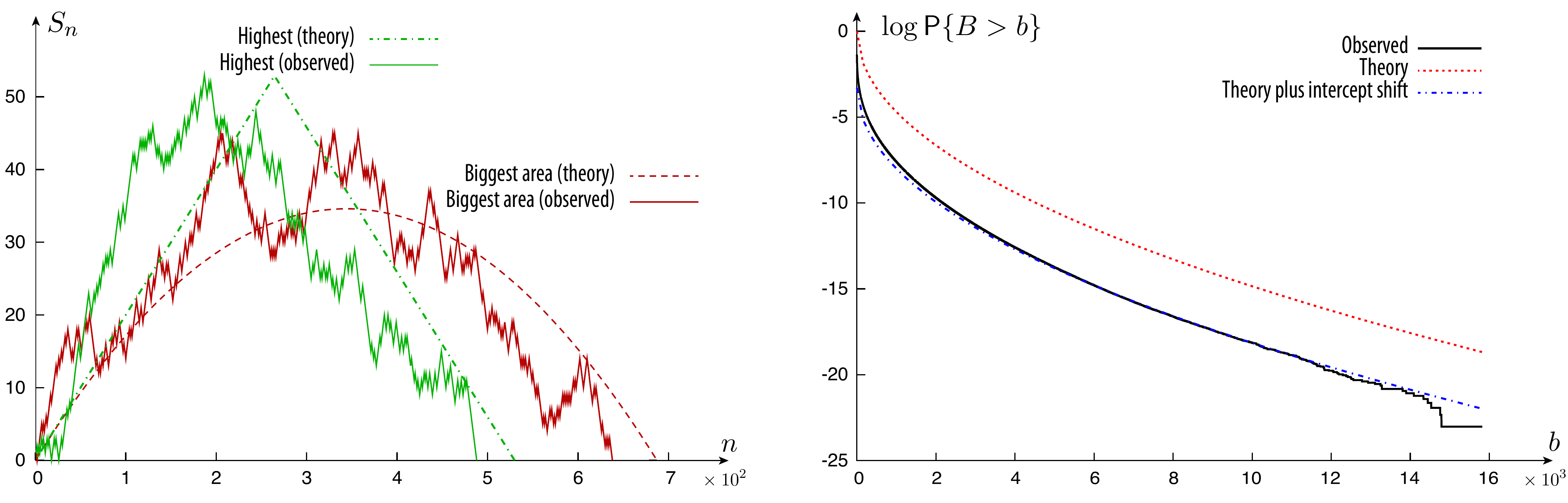}
\caption{Bernoulli $\{-1,+1\}$ increments 
with $P(X_1=-1)=0.6$, $P(X_1=+1)=0.4$. Simulated results from ten
billion paths.  (Left panel) Four paths displayed. Simulated path
with largest swept area $15,825$ and, conditioned on that area,
most likely path to give rise to it predicted by \Thm{thm:main},
which is concave. For contrast, also shown is the simulated path
that reaches the greatest height, $53$, and, conditioned on that
height, the most likely path to give rise to it as deduced from
\Thm{thm:goc}, which is piecewise linear.  (Right panel) Logarithm
of the empirical probability that area under busy period exceeds a
given level compared with $-K\sqrt{b}$ and, the linear
shift $-K\sqrt{b}+\kappa$ where $\kappa$ is chosen to match
the offset of the empirical observation}
\label{fig:mm1}
\end{center}
\end{figure}

\begin{figure}
\begin{center}
\includegraphics[width=1.00\columnwidth]{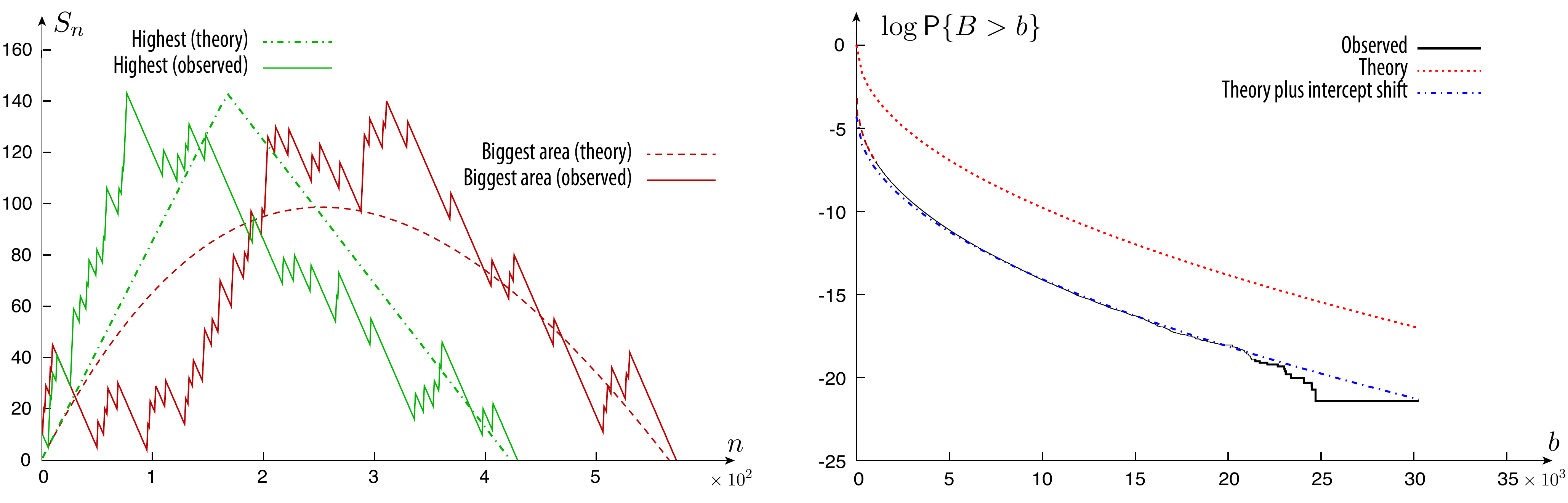}
\caption{Bernoulli $\{-1,+10\}$ increments with
$P(X_1=-1)=0.96$, $P(X_1=+10)=0.04$. 
Simulated results from ten billion paths. 
(Left panel)
Four paths displayed. Simulated path with largest swept area
$36,971$ and, conditioned on that area, most
likely path to give rise to it predicted by \Thm{thm:main},
which is concave. For contrast, also shown is the simulated path
that reaches the greatest height, $143$, and, conditioned on that
height, the most likely path to give rise to it as deduced from
\Thm{thm:goc}, which is piecewise linear.
(Right panel) Logarithm of the empirical probability that area
under busy period exceeds a given level compared with
$-K\sqrt{b}$ and, the linear shift  $-K\sqrt{b}+\kappa$ 
where $\kappa$ is chosen to match the offset of the empirical
observation}
\label{fig:bern10}
\end{center}
\end{figure}

Let $\bfmX$ be a Bernoulli sequence taking values $-1$ and $C$ with
$\alpha=1-\mu=P\{X_0=C\}<P\{X_0=-1\}=\mu$, so that $-\delta=\alpha C-\mu$,
and define the load $\rho=\alpha C/\mu$. For this process
\begin{align*}
\Lambda(\theta) = \log\left(\mu e^{-\theta}+\alpha e^{C\theta}\right)
\end{align*}
and, by \eq{e:xstar}, $\lambda^*$ is the unique positive
solution of 
\begin{align*}
\alpha e^{\lambda^*C} +\mu e^{\lambda^*}-1=0.
\end{align*}
Closed form expressions for $\lambda^*$ do not exist apart from in
a few special cases. If $C=1$, $\Lambda$ satisfies the conditions
of Proposition~\ref{lemma:rev}, the most likely path is symmetric
and this equation has the explicit solution $\lambda^*=-\log(\rho)$,
which could also have been determined by $2\nabla I(0)$ or
$I(\delta)/\delta$. Given the reversibility of the M/M/1 queue,
the symmetry of $\psi^*$ established in Proposition~\ref{lemma:rev}
is not surprising.

If $C\ge 2$, then the corresponding reflected random walk is not
reversible, and moreover the increments do not satisfy the symmetry
assumptions of Proposition~\ref{lemma:rev}.
For $C=2$ we obtain the expression,
\begin{align*}
\lambda^* = \log\left(
	\frac{\sqrt{- 3\mu^2 + 2\mu + 1} - \alpha}{2\alpha}
\right),
\end{align*}
while if $C=3$
\begin{align*}
\lambda^* &= \log\left(\frac{7z}{9}  -\frac{1}{3}\right),
\text{ where }
z= \sqrt[3]{\sqrt{\left(\frac{\mu}{2(\mu - 1)} - \frac{7}{54}\right)^2 
	+ \frac{8}{729}}
	- \frac{\mu}{2(\mu - 1)} + \frac{7}{54}}.
\end{align*}

Using \eq{eq:psi*alt1}, for arbitrary $C$ we conclude
that the most likely path on $[0,1]$ can be written in terms of
$\lambda^*$ as
\begin{align*}
\psi^*(t) 
	&= -\frac{1}{\lambda^*} 
		\log\left(\mu e^{-\lambda^*(1-t)} + \alpha e^{C\lambda^*(1-t)}
		\right).
\end{align*}
In order to calculate $K$ in equation \eqref{e:Kformula}, we need to
evaluate the integral $\int_0^{\lambda^*} \Lambda(\theta) \, d\theta$.
This doesn't result in a closed form for any $C$, but it is simple
to evaluate numerically.

In order to determine the most likely time and paths to a large
busy period and a great height on the scale of large deviations we
use the approximations \eq{eq:bigpathapprox} and \eq{eq:highpathapprox}.
For the reversible Bernoulli $\{-1,+1\}$ case corresponding to M/M/1
queue-lengths, \Fig{fig:mm1} compares the highest and
biggest paths with those from theory. The quality of the predictions
is apparent. With a numerical integration of $\Lambda$ giving
$K\approx 0.1485$, the asymptotic approximation is compared
with the empirical probability, showing great accuracy up to a
constant prefactor.

As an example of a non-reversible random walk, we consider the
Bernoulli $\{-1,+10\}$ case where the most likely paths are now
asymmetric (as seen in \Fig{fig:bern10}). For this example, $\lambda^*\approx
0.1439$ and $K\approx 0.0978$, both of which have been determined
numerically.

\subsection{Markovian $\{-1,+1\}$ increments}

As an example beyond i.i.d.\ increments, assume that the increments
process $\bfmX$ forms a two-state Markov chain on the state space
$\{-1,+1\}$ with transition matrix
\begin{align*}
\left( 
        \begin{array}{cc}
        1-\alpha & \alpha\\
        \beta & 1-\beta 
        \end{array}
        \right)
\text{ where } 0<\alpha < \beta<1. 
\end{align*}
The stationary distribution is $(\beta/(\alpha+\beta),\alpha/(\alpha+\beta))$
so we require $\alpha<\beta$ for stability. The sCGF $\Lambda$ can
be calculated using techniques described in \cite[Section 3.1]{Dembo98}:
\begin{align*}
\label{eq:2statescgf}
\Lambda(\theta) 
 = \log\left(\frac{(1-\alpha)e^{-\theta} 
  + (1-\beta)e^{\theta} 
  + \sqrt{4\alpha\beta+((1-\alpha)e^{-\theta}
	-(1-\beta)e^{\theta})^2})}{2}\right),
\end{align*}
which satisfies the conditions of Proposition~\ref{lemma:rev} and so
the most likely path is symmetric.
The rate function for the associated random walk can be calculated
using methods described in \cite{Deuschel89}. For example, from
\cite{Duffy05} we have that
\begin{align*}
I(x) = -\left(\frac{1-x}{2}\right)\log(1-\alpha+\alpha \chi(x))
	 -\left(\frac{1+x}{2}\right)\log(1-\beta+\beta/\chi(x))
\end{align*}
where
\begin{align*}
\chi(x) = \frac{ \alpha\beta x 
	+\sqrt{\alpha^2\beta^2x^2+\alpha\beta(1+x)(1-\alpha)(1-\beta)(1-x)}
	}{\alpha(1-\beta)(1-x)}.
\end{align*}
One can check directly that
\begin{align*}
x^* = -\rho = \delta = \frac{\beta-\alpha}{\alpha+\beta}
\text{ and } 
\lambda^* = I(x^*)/x^* 
	= \nabla I(x^*) 
	= 2\nabla I(0) 
	= \sup\{\theta:\Lambda(\theta)\leq0\}
	= \log\left(\frac{1-\alpha}{1-\beta}\right).
\end{align*}
We have the following expression of the most likely path returning
to $0$ at $t=1$, which can be seen directly to possess the symmetry
$\psi^*(t)=\psi^*(1-t)$ for $t\in[0,1]$,
\begin{align*}
&\psi^*(t)=-\frac{1}{\log(\frac{1-\alpha}{1-\beta})} \\
 &\log\left( 
	\frac{
	(1-\alpha)^t (1-\beta)^{1-t}
	+(1-\alpha)^{1-t}(1-\beta)^t 
  + \sqrt{4\alpha\beta+(
	(1-\alpha)^t (1-\beta)^{1-t}
	-(1-\alpha)^{1-t}(1-\beta)^t)^2})}{2}\right).
\end{align*}
The integral $\int_0^{\lambda^*} \Lambda(\theta)\,d\theta$ does not
evaluate in closed form, so again numerics must be used to determine
$K$. For example, if $\alpha=2/10$ and $\beta = 3/10$, then $\delta
= 2/10$, $x^*=2/10$, $\lambda^*=\log(8/7)$ and $K\approx 0.0489$.
A simulation-based illustration of these results appears in
\Fig{fig:markov}.

\begin{figure}
\begin{center}
\includegraphics[width=1.00\columnwidth]{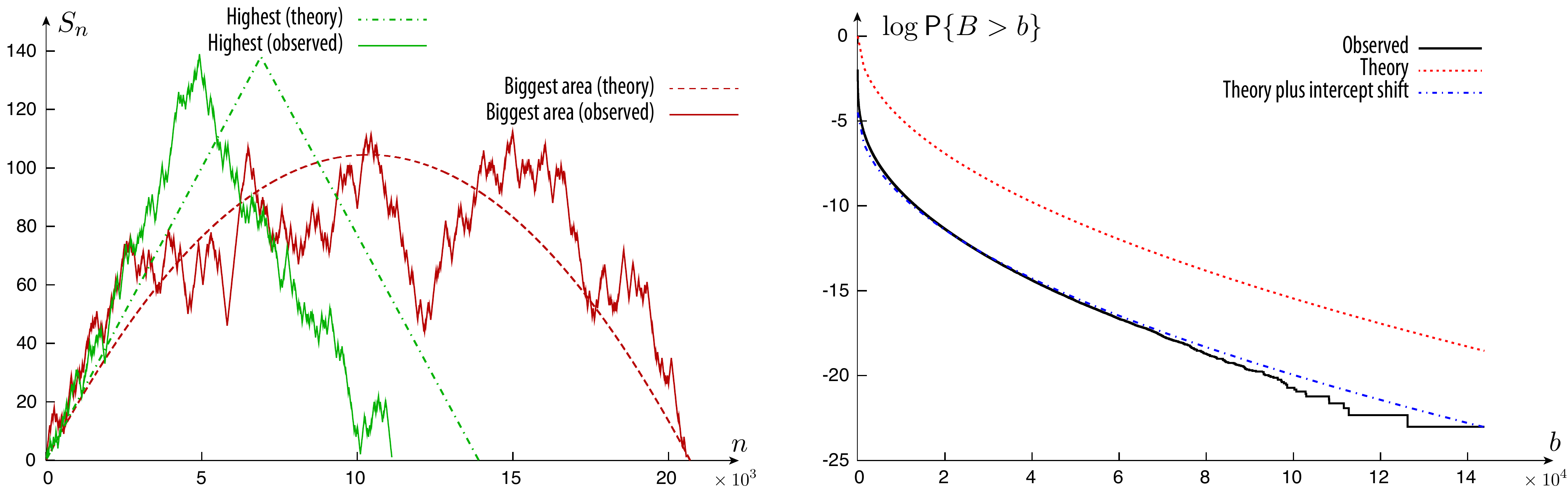}
\caption{Markov $\{-1,+1\}$ increments with
$P(X_{t+1}=+1|X_t=-1)=\alpha=2/10$, $P(X_{t+1}=-1|X_t=+1)=\beta=3/10$. 
Simulated results from ten billion paths. 
(Left panel)
Four paths displayed. Simulated path with largest swept area
$143,875$ and, conditioned on that area, most
likely path to give rise to it predicted by \Thm{thm:main},
which is concave. For contrast, also shown is the simulated path
that reaches the greatest height, $139$, and, conditioned on that
height, the most likely path to give rise to it as deduced from
\Thm{thm:goc}, which is piecewise linear.
(Right panel)
Logarithm of the empirical probability that area under busy period
exceeds a given level compared with $-K\sqrt{b}$ and,
the linear shift $-K\sqrt{b}+\kappa$ where
$\kappa$ is chosen to match the offset of the empirical observation}
\label{fig:markov}
\end{center}
\end{figure}

\section{Conclusions and discussion on estimation}
\label{sec:conc}

As well as identifying the most likely path to a large swept area,
\Thm{thm:main} shows that in the absence of long range dependence
and heavy tailed increments, in broad generality we have the
approximation in \eq{eq:rfapprox}:
\begin{align*} 
\log P(B\geq b) \approx -K\sqrt{b}, \text{ for large } b,
\end{align*}
where $K$ can be identified in terms of the sCGF $\Lambda$ associated
with the increments $\bfmX$. An approximation such as this might be of
value for practical purposes, but unless $\bfmX$ is known in advance
we would require a methodology to estimate $K$ from observations
of the system.

Based on thermodynamic ideas, Duffield {\it et al.} \cite{Duffield95A}
investigated an estimation scheme for $\Lambda$ based on observations
of $\bfmX$. They demonstrated empirically that the scheme has
desirable properties for a large class of of increment processes.
Indeed, if $\bfmX$ consists of i.i.d. bounded random variables
\cite[Theorem 1]{Duffy05} or a finite state Markov chain \cite[Theorem
3]{DuffyMeyn11}, from observations of $\bfmX$ one can construct
consistent functional estimates, $\{\estLambda_n\}$, of
$\Lambda$ that themselves satisfy a LDP in the space of $\R$-valued
convex functions on $\R$. From these, we can deduce an LDP for
estimating $K$ from observations of $\bfmX$ as follows.
If $\bfmX$ is i.i.d, define
\begin{align*}
\estLambda_n(\theta) = \log\left(\frac1n\sum_{k=1}^ne^{\theta X_k}\right).
\end{align*}
If $\bfmX$ forms a finite state Markov chain with an irreducible
transition matrix on $\{f(1),\ldots,f(M)\}$ where $f(i)\neq f(j)$
for $i\neq j$, then with $0/0$ defined to be $0$ we define an
empirical transition matrix  with entries,
\begin{align*}
(\estP_n)_{i,j} :=  
        \left(\sum_{k=1}^n 1_{\{(X_{k-1},X_k)=(f(i),f(j))\}} \right) /
                \left(\sum_{k=1}^n 1_{\{X_{k-1}=f(i)\}}\right),
\end{align*}
and let $D_\theta$ denote the matrix with diagonal entries
$\exp(\theta f(1)), \ldots, \exp(\theta f(M))$
and all off-diagonal entries equal to zero. Then our estimate 
of $\Lambda$ given $n$ observations is
\begin{align*}
\estLambda_n(\theta) = \log\rho(\estP_nD_\theta),
\end{align*}
where where $\rho$ is the spectral radius. In both cases we define
the estimates
\begin{align*}
\lambda^*_n = \sup\{\theta:\estLambda_n(\theta)\leq 0\}
\text{ and }
K_n = 2 \sqrt{-\int_0^{\lambda^*_n}\estLambda_n(\theta)\,d\theta}.
\end{align*}

Regarding the $\lambda^*$ estimates, \cite[Lemma 1]{Duffy05} proves
that the function $g(c)=\sup\{\theta:c(\theta)\leq0\}$ is continuous
on the space of convex functions equipped with the topology of
uniform convergence on compact subsets at all $c$ such that $c(0)=0$
and there does not exist a $\kappa>0$ so that $c(\theta)=0$ for
$\theta\in[0,\kappa]$.  Thus from the LDP for the estimators
$\{\estLambda_n\}$, if the support of $X_k$ excludes a finite ball
around the origin and if $f(k)\neq0$ for any $k$ in the Markovian
case, we have an LDP for $\{\lambda^*_n\}$ by Puhalskii's extension
of the contraction principle \cite[Theorem 2.1]{Puhalskii95}.
As the convex functions of interest are real-valued, the function
$c\mapsto\int_0^{g(c)}c(\theta)d\theta$ is also continuous so
that again Puhalskii's extension of the contraction principle applies
and we have an LDP for $\{K_n\}$. Consistency of the $\Lambda$
estimators ensures consistency of the $\lambda^*$ and $K$ estimators.
Thus these $\Lambda$-estimators enable the estimation of $K$ directly
from observations of $\bfmX$, offering the possibility of estimating
the probability of a system experiencing an exceedingly large busy
period in advance of one occurring.

\bigskip
\noindent
\textbf{Acknowledgment}
Financial support from the AFOSR
grant FA9550-09-1-0190  is gratefully
acknowledged.

\def\cprime{$'$}


\end{document}